\documentclass[12pt]{amsart}

\usepackage{amsfonts,amsmath,amssymb,amsthm,color,hyperref,tikz}
\usepackage{mathrsfs}

\usepackage{enumitem}

\usepackage[margin=1in]{geometry}

\theoremstyle{definition}
\newtheorem{thm}{Theorem}[section]
\newtheorem{obs}[thm]{Observation}
\newtheorem{prop}[thm]{Proposition}

\newtheorem{cor}[thm]{Corollary}
\newtheorem{example}[thm]{Example}
\newtheorem{conj}[thm]{Conjecture}
\newtheorem{notation}[thm]{Notation}
\newtheorem{lem}[thm]{Lemma}

\newtheorem{problem}[thm]{Problem}

\DeclareMathOperator{\img}{im}
\DeclareMathOperator{\NVol}{NVol}
\DeclareMathOperator{\vol}{vol}
\DeclareMathOperator{\cat}{cat}
\DeclareMathOperator{\conv}{conv}
\DeclareMathOperator{\Ehr}{Ehr}

\DeclareMathOperator{\PM}{PM}
\DeclareMathOperator{\PQ}{PQ}

\newcommand{\RR}{\mathbb{R}}
\newcommand{\ZZ}{\mathbb{Z}}

\newcommand{\APQ}{\nabla^{\PQ}}

\newcommand{\calM}{\mathcal{M}}

\newcommand{\calS}{\mathcal{S}}

\newcommand{\h}{\overline{h}}

\newcommand\commentout[1]{}

\begin{document}

\title[Perfectly Matchable Set Polynomials and $h^*$-polynomials]{Perfectly Matchable Set Polynomials and $h^*$-polynomials for Stable Set Polytopes of Complements of Graphs}

\author{Robert Davis}
\address{Department of Mathematics\\
         Colgate University\\
         Hamilton, NY USA}
\email{rdavis@colgate.edu}

\author{Florian Kohl}
\address{Department of Mathematics and Systems Analysis\\
         Aalto University\\
         Espoo, Finland}
\email{florian.kohl@aalto.fi}

\thanks{The first author was supported in part by NSF grant DMS-1922998, and the second author was supported by Academy of Finland project 324921.}

\begin{abstract}
    A subset $S$ of vertices of a graph $G$ is called a perfectly matchable set of $G$ if the subgraph induced by $S$ contains a perfect matching.
    The perfectly matchable set polynomial of $G$, first made explicit by Ohsugi and Tsuchiya, is the (ordinary) generating function $p(G; z)$ for the number of perfectly matchable sets of $G$.
    
    In this work, we provide explicit recurrences for computing $p(G; z)$ for an arbitrary (simple) graph and use these to compute the Ehrhart $h^*$-polynomials for certain lattice polytopes.
    Namely, we show that $p(G; z)$ is the $h^*$-polynomial for certain classes of stable set polytopes, whose vertices correspond to stable sets of $G$.
\end{abstract}

\maketitle

\tableofcontents


\section{Introduction}

For a graph $G$, we let $V(G)$ denote the \emph{vertices} of $G$ and let $E(G)$ denote the $edges$ of $G$.
This article considers two combinatorial structures associated to graphs, which we always assume to be finite and simple.
The first structure relies on perfectly matchable sets: a \emph{perfectly matchable set} in $G$ is a subset $S \subseteq V(G)$ such that $G[S]$, the subgraph induced by $S$, contains at least one perfect matching.
We use the convention that the empty set is a perfectly matchable set.
The \emph{perfectly matchable set polynomial} of $G$ is
\[
    p(G; z) = \sum_{k=0}^{\left\lfloor\frac{|V(G)|}{2}\right\rfloor} p_{2k}z^k
\]
where $p_{2k}$ is the number of perfectly matchable sets in $G$ of size $2k$. 
For example, for the $4$-cycle $C_4$ we have $p(C_4; z) = 1+4z+z^2$.
Note that even though $C_4$ admits two distinct perfect matchings, they each use the same set of vertices, hence the summand $z^2$ instead of $2z^2$.

For the second structure, recall that a \emph{lattice polytope} $P \subseteq \RR^n$ is the convex hull of a finite set of points $v_1,\dots,v_k \in \ZZ^n$ and is often written 
\[
    P = \conv\{v_1,\dots,v_k\} = \left\{ \sum_{i=1}^k \lambda_i v_i \mid \lambda_i \geq 0,\, \sum_{i=1}^k \lambda_i = 1\right \}.
\]
Ehrhart \cite{Ehrhart} showed that the function $t \mapsto |tP \cap \ZZ^n|$ agrees with a polynomial $L_P(t)$ of degree $\dim(P)$, called the \emph{Ehrhart polynomial} of $P$.
In turn, the generating function
\[
    \Ehr_P(z) = 1 + \sum_{m \geq 1} L_P(m)z^m,
\]
called the \emph{Ehrhart series} of $P$, has rational form
\[
    \Ehr_P(z) = \frac{h^*(P; z)}{(1-z)^{\dim(P)+1}}
\]
where $h^*(P; z)$ is a polynomial of degree $d \leq \dim(P)$, called the \emph{$h^*$-polynomial} of $P$.
The sequence $h^*(P) = (h_0^*,\dots,h_d^*)$ is the \emph{$h^*$-vector} of $P$.

The $h^*$-polynomials of polytopes have been the subject of intense scrutiny in recent years, since they encode various data of the corresponding polytopes.
For example, the \emph{normalized volume} of $P$, $\NVol(P) = (\dim P)!\vol(P)$ where $\vol$ is the relative Euclidean volume of $P$, can be computed indirectly as $\NVol(P) = h^*(P;1)$.
Additionally, the linear coefficient of $h^*(P; z)$ can be shown to be, through elementary methods, $|P \cap \ZZ^n| - \dim(P)-1$. 
See \cite{BraunUnimodality} for a recent survey of work that has been done in understanding the coefficients of $h^*(P; z)$. 

We will focus on $h^*(P; z)$ for a special class of polytopes.
Recall that a \emph{stable set} of a graph $G$ is a subset $S \subseteq V(G)$ such that no two elements of $S$ are adjacent. 
Stable sets are sometimes called \emph{independent sets}, but we prefer ``stable sets'' in order to avoid potential conflict with the notion of independence in matroid theory. 
Let $\calS_G$ be the collection of stable sets of $G$.
The \emph{stable set polytope} $P_G$ is defined as
\[
	P_G = \conv\{\chi_G(S) \in \RR^{|V(G)|} \mid S \in \calS_G\}
\]
where $\chi_G = \chi_{V(G)}$ is the indicator vector.

In this article, we examine $h^*(P_G; z)$ for large classes of graphs $G$ using perfectly matchable set polynomials.
Section~\ref{sec: background} provides context for where this work originated and how it can provide insight for ongoing work regarding a certain subpolytope of $P_G$ that arises in the study of power-flow equations \cite{ChenMehta, DavisChenDraconian, DavisJakovleskiPan}.
In Section~\ref{sec: recurrences} we give several recurrences for $p$ (Proposition~\ref{prop: p(G; z) with no even cycles}, Theorem~\ref{thm: perf match sets recurrence}, Theorem~\ref{thm: cut vertex}). 
With these results and their consequences, we show that, for any $G$, $p(G; z)$ can be computed by a applying a sequence of recurrences. 
In Section~\ref{sec: relation to Ehrhart} we illustrate that $p(G; z)$ is often an $h^*$-polynomial of a related stable set polytope, depending on properties of the graph.
We close with Section~\ref{sec: special cases}, in which we more closely examine $p(G; z)$, in terms of $h^*$-polynomials, for several classes of trees.
Even in this highly restrictive case, we obtain interesting $h^*$-polynomials. 

\section{Background}\label{sec: background}

This work is motivated in part by the study of type-PQ adjacency polytopes, introduced by Chen and the first author \cite{DavisChenDraconian} and extended in \cite{DavisJakovleskiPan}, which arise in the study of power-flow equations for electrical networks.
Namely, if $G$ is a graph on the set $[n] = \{1,\dots,n\}$, then the \emph{type-PQ adjacency polytope} is
\[
    \APQ_G = \conv\{(e_i, e_j) \in \RR^{2n} \mid ij \in E(G) \text{ or } i=j\}.
\]
For a graph $G$ on $[n]$, let $[\overline n] = \{\overline 1, \dots, \overline n\}$ and let $D(G)$ be the graph on $[n] \cup [\overline n]$ where all edges are of the form $i\overline j$, and $i \overline j$ is an edge if and only if $i = j$ or if $ij = ji$ is an edge in $G$. 
See Figure~\ref{fig: D(G)} for an example.
One may also consider $D(G)$ as the bipartite double cover of the graph obtained from $G$ by adding exactly one loop to each vertex.
On the polytope side, one may consider $\APQ_G$ as the edge polytope of $D(G)$ or as a subpolytope of the $(|V(G)|,2)$-hypersimplex, i.e., the convex hull of all indicator vectors $\chi_{V(G)}(S)$ with $S \in \binom{V(G)}{2}$.

\begin{figure}
\begin{center}
\begin{tikzpicture}
\begin{scope}[every node/.style={circle,fill,inner sep=0pt,minimum size=2mm}]
	\node[label=$1$] (A) at (-1,0) {};
	\node[label=$2$] (B) at (0.8,0) {};
	\node[label=right:$3$] (C) at (2,1) {};
	\node[label=right:$4$] (D) at (2,-1) {};
\end{scope}
\node (E) at (0,-1.5) {};

\draw[thick] (A) -- (B) -- (C) -- (D) -- (B);
\end{tikzpicture}
\hspace{1in}
\begin{tikzpicture}

\begin{scope}[every node/.style={circle,fill,inner sep=0pt,minimum size=2mm}]
	\node[label=left:$1$] (A) at (0,0) {};
	\node[label=left:$2$] (B) at (0,1) {};
	\node[label=left:$3$] (C) at (0,2) {};
	\node[label=left:$4$] (D) at (0,3) {};
	\node[label=right:$\overline{1}$] (AA) at (1,0) {};
	\node[label=right:$\overline{2}$] (BB) at (1,1) {};
	\node[label=right:$\overline{3}$] (CC) at (1,2) {};
	\node[label=right:$\overline{4}$] (DD) at (1,3) {};
\end{scope}

\draw[thick] (A) -- (AA) -- (B) -- (BB) -- (A);
\draw[thick] (B) -- (CC) -- (C) -- (BB) -- (D) -- (CC);
\draw[thick] (D) -- (DD) -- (C);
\draw[thick] (B) -- (DD);
\end{tikzpicture}
\end{center}
\caption{A graph $G$, left, and its corresponding bipartite graph $D(G)$, right.}\label{fig: D(G)}
\end{figure}
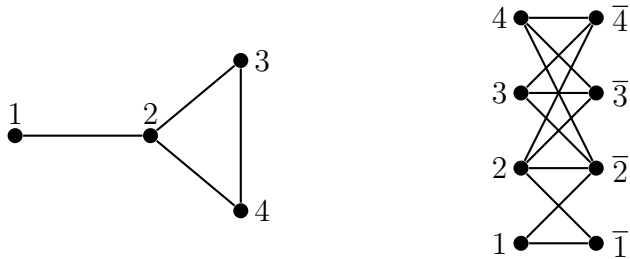

Thus, there is no reason to expect simple, exact, algebraic formulas for $\NVol(\APQ(G))$ in general.
However, considering $\APQ_G$ as a subset of another polytope does allow one to obtain bounds on not only the normalized volume but also on the coefficients of the $h^*$-polynomial.
This relies in part on the fact that $h^*(P,1) = \NVol(P)$ for every lattice polytope $P$; see \cite[Corollary 3.21]{BeckRobinsCCDed2}.

\begin{thm}[{\cite[Theorem 3.3]{StanleyMonotonicity}}]\label{thm: monotonicity}
    Let $P, Q \subseteq \RR^n$ be lattice polytopes such that $P$ is a subpolytope of $Q$.
    Let $h^*(P) = (h_0^*(P),h_1^*(P), \dots, h_d^*(P))$ be the $h^*$-vector of $P$ and let $h^*(Q) = (h_0^*(Q),h_1^*(Q), \dots, h_{d'}^*(Q))$ be the $h^*$-vector of $Q$.
    Then $d \leq d'$ and $h_i^*(P) \leq h_i^*(Q)$ for all $i = 1,\dots, d$.
\end{thm}

Suppose $P \subseteq \RR^n$ is a polytope.
A hyperplane $H$ given by the linear equation $\ell(x) = b$ defines a \emph{face} if either $\ell(x) \leq b$ for all $x \in P$ or $\ell(x) \geq b$ for all $x \in P$.
A \emph{facet} of $P$ is a face of dimension $\dim(P)-1$.
Each face of $P$ is an intersection of some set of facets of $P$.
Lastly, recall that a \emph{clique} in a graph is a set $S \subseteq V(G)$ for which $G[S]$ is complete.

\begin{thm}[{\cite{Padberg}}]
    Let $G$ be a graph.
    If $S$ is a maximal clique in $G$, then the equation $\langle \chi_S, x\rangle = 1$ describes a facet-defining hyperplane of $P_G$.
\end{thm}

For our connection, recall that for a graph $G$, the \emph{complement} of $G$ is denoted $\overline{G}$.
When $G$ is bipartite, then the maximal cliques of $G$ are exactly the edges; equivalently, the maximal stable sets of $\overline{G}$ are the edges of $G$. 
This brings us to the observation which takes our focus from type-PQ adjacency polytopes to stable set polytopes.

\begin{obs}\label{obs: PQ face of stable}
	For any graph $G$, $\APQ_G$ is a face of $P_{\overline{D(G)}}$.
\end{obs}

Putting Theorem~\ref{thm: monotonicity} together with Observation~\ref{obs: PQ face of stable}, we get the following.

\begin{lem}
    For any graph $G$, $h_i^*(\APQ_G) \leq h_i^*(P_{\overline{D(G)}})$
    for all $i$, and therefore $\NVol(\APQ_G) \leq \NVol(P_{\overline{D(G)}})$.
\end{lem}

For often-loose but reasonable lower- and upper- bounds on $\NVol(\APQ_G)$ connected graphs, we can simply use 
\[
    \APQ_T \subseteq \APQ_G \subseteq \APQ_{K_n}
\]
where $T$ is a spanning tree of $G$ and $n = |V(G)|$.
Corollary 14 of \cite{DavisChenDraconian} gives $\NVol(\APQ_T) = 2^{n-1}$, and Theorem 2.3 of \cite{GGP} gives $\NVol(\APQ_{K_n}) = \binom{2(n-1)}{n-1}$.
These bounds can be refined in terms of $h^*$-numbers:
\begin{equation}\label{eq: h-star bounds}
    \binom{n-1}{i} \leq h^*_i \leq \binom{n-1}{i}^2
\end{equation}
where $h^*_0, \dots, h^*_n$ is the $h^*$-vector for $\APQ_G$.
The lower bound is due to upcoming work of the first author \cite{DavisInPrep} and the upper bound is Theorem 13 of \cite{ArdilaGrowthLattices}.

A first step towards understanding more about the $h^*$-vector for $\APQ_G$ would be to utilize the connection to stable set polytopes of perfect graphs. 
While the the hyperplane description of stable set polytopes is in general not known, they are known in the important special case of perfect graphs. 
A graph $G$ is called \emph{perfect} if, for every induced subgraph $G'$, the chromatic number of $G'$ is the maximum size of a clique of $G'$.
The following result, due to Lov\'asz~\cite{Lovasz} gives not only the full hyperplane description, but it even characterizes perfect graphs in terms of stable set polytopes.

\begin{thm}[{\cite{Lovasz}}]\label{thm:Lovasz}
    A graph $G$ on $[n]$ is perfect if and only if 
	\[
        P_G \ = \ \{ x \in \RR^n_{\geq 0}   \mid \ \langle \chi_C,x\rangle \leq 1
        \text{ for all maximal cliques }  C\subseteq [n] \}.
	\]
\end{thm} 	

This result is relevant for us, since bipartite graphs are perfect. By the weak perfect graph theorem, their complements are also perfect.

\begin{cor}[Weak perfect graph theorem, \cite{Lovasz}]
    A graph $G$ is perfect if and only if $\overline G$ is perfect.
\end{cor}

Therefore, $P_{\overline{D(G)}}$ is the stable set polytope of a perfect graph, which in particular induces further constraints on the associated $h^{\ast}$-polynomial.
In fact, \eqref{eq: h-star bounds} implies that, for any connected graph $G$, $h^*_0 = h^*_{n-1} = 1$. 
In light of these bounds, we pose the following problem.

\begin{problem}
    Determine the sharpest possible lower- and upper-bounds on each $h_i^*$ for $\APQ_G$ for large classes of graphs.
\end{problem}

It is for this reason, together with Observation~\ref{obs: PQ face of stable}, that we consider stable set polytopes.

Our interest in perfectly matchable set polynomials is due to the success found by Ohsugi and Tsuchiya in using them to determine $h^*(\APQ_G)$ when $G$ is a \emph{wheel} graph, i.e., $G$ is the join of a cycle and a single vertex \cite{OhsugiTsuchiyaPQ}. 
The notion of a perfectly matchable set is originally due to Balas and Pulleyblank \cite{BalasPulleyblankGeneral, BalasPulleyblankBipartite} in the context of linear optimization over polytopes.
Their work has since been built upon, e.g. \cite{CunninghamGreen-Krotki, Sharifov}, but the perfectly matchable set polynomial $p(G; z)$ itself was not introduced until Ohsugi and Tsuchiya themselves in \cite{OhsugiTsuchiyaInterior}.

These results, and the desire to understand when the bounds of \eqref{eq: h-star bounds} can be sharpened, provide the motivation for our work.


\section{Recurrences for perfectly matchable set polynomials}\label{sec: recurrences}

Recall that a \emph{matching} of $G$ is a set $\calM$ of edges in $G$ such that no two distinct edges in $\calM$ share a common endpoint.
Let $m(G; z)$ denote the \emph{matching polynomial} $m(G; z) = \sum_{k \geq 0} m_kz^k$, where $m_k$ is the number of matchings of size $k$ in $G$.
The matching polynomial of $G$ is a specialization of the bivariate polynomial
$M(G;x,y) = \sum_{k \geq 0} a_kx^{n-2k}y^k$ where $n$ is the number of vertices of $G$.
This specialization is given by
\begin{equation}\label{eq: specialization}
	m(G; z) = z^nM\left(G;\frac{1}{z},\frac{1}{z}\right).
\end{equation}
The polynomials $p(G; z)$ and $m(G; z)$ are genuinely distinct: returning to the example of $C_4$, we have $m(C_4; z) = 1 + 4z + 2z^2$ while $p(G; z) = 1 + 4z + z^2$. 

There are two recurrences, due to Farrell \cite{Farrell}, which the function $M$ satisfies and which will be of interest to us as we study $p(G; z)$.
If $e$ is an edge of $G$, we denote by $G \setminus e$ the deletion of $e$ from $G$.
To help distinguish between deleting an edge and deleting a vertex $v$, we use $G - v$ to denote the deletion of a vertex, or $G - \{v,w\}$ if two vertices, $v$ and $w$, are deleted.

\begin{thm}[{\cite[Theorems 1 and 2]{Farrell}}]\label{thm: farrell}
	Let $G$ be any graph.
	\begin{enumerate}
		\item If $e = vw$ is an edge of $G$, then
			\[
				M(G; x,y) = M(G \setminus e;x,y) + yM(G - \{v,w\}; x,y)
			\]
		\item If $v$ is a vertex of $G$ with neighbors $w_1,\dots,w_r$, then
			\[
				M(G; x,y) = xM(G - v; x,y) + y\sum_{i=1}^r M(G - \{v,w_i\};x,y).
			\]
	\end{enumerate}
\end{thm}

Ohsugi and Tsuchiya noted \cite[Example 2.3]{OhsugiTsuchiyaPQ} the following result, which has a straightforward proof. 
Since we wish to use this result and could not find an explicit proof of it within the existing literature, we present it below.

\begin{lem}\label{lem: no even cycles}
    For a graph $G$, $p(G; z) = m(G; z)$ if and only if $G$ has no even cycles.
\end{lem}

\begin{proof}
    First, if $p(G; z) = m(G; z)$, then assume an even cycle $C$ with $k$ edges exists in $G$. 
    The subgraph of $G$ induced by the vertices of $C$ must then have at least two perfect matchings, since each even cycle contains exactly two. 
    Thus, the coefficient of $z^k$ in $m(G; z)$ is strictly greater than the corresponding coefficient in $p$, which is a contradiction.
    
    Conversely, suppose $G$ has no even cycles. 
    If there is an induced subgraph $G[S]$ with $2k$ vertices which contains two perfect matchings $\calM$ and $\calM'$, then each vertex in the subgraph $H = (S,\calM \cup \calM')$ has degree $1$ or $2$, hence $H$ is a disjoint union of paths and/or cycles.
    
    If $P$ is a path in $H$, then let $e$ be an edge with an endpoint $u$ of degree $1$ in $H$. 
    Then $u$ is covered by exactly one edge in each of $\calM$ and $\calM'$, so $e$ is in both matchings.
    It follows that $e$ is the only edge of $P$.
    Therefore, any path in $H$ consists of a a single edge that is in both $\calM$ and $\calM'$.
    
    Now, since $\calM$ and $\calM'$ are distinct perfect matchings, there must be at least one cycle in $H$.
    However, the edges in this cycle must alternate with edges from $\calM$ and $\calM'$, so the cycle is even, which is a contradiction.
    Therefore, every subgraph of $G$ induced by an even number of vertices has a unique perfect matching.
    Hence, $p(G; z) = m(G; z)$.
\end{proof}

Putting Lemma~\ref{lem: no even cycles} together with Theorem~\ref{thm: farrell} and \eqref{eq: specialization}, we obtain the following.

\begin{prop}\label{prop: p(G; z) with no even cycles}
    If $G$ is a graph with no even cycles and $e = vw$ is an edge of $G$, then
    \[
        p(G; z) = p(G \setminus e; z) + zp(G - \{v,w\}; z).
    \]
    Moreover, if $v$ is a vertex of $G$ with neighbors $w_1,\dots,w_r$, then
	\[ \pushQED{\qed}
		p(G; z) = p(G - v; z) + z\sum_{i=1}^r p(G - \{v,w_i\}; z).
		\qedhere\popQED
	\]
\end{prop}

The following is a straightforward consequence.
In it, we will use a specialization of a bivariate polynomial introduced by Hoggatt and Long \cite{HoggattLong}: set $u_0(x,y) = 0$, $u_1(x,y)=1$, and 
\[
    u_n(x,y) = xu_{n-1}(x,y) + yu_{n-2}(x,y).
\]
for $n \geq 2$.
Note that the \emph{Fibonacci polynomials} $F_n(x)$ are exactly those polynomials obtained by setting $y=1$.
We also use the notation $P_n$ to denote the path on $n$ vertices.

\begin{cor}\label{cor: paths and cycles}
    For all $n \geq 1$, $p(P_n; z) = u_{n+1}(1,z)$ 
    and for $n \geq 3$,
    \[
            p(C_n; z) = 
                \begin{cases}
                    u_{n+1}(1,z) + zu_{n-1}(1,z) & \text{ if } n \text{ odd}\\
                    u_{n+1}(1,z) + zu_{n-1}(1,z) - z^{n/2} & \text{ if } n \text{ even}.
                \end{cases}
    \]
\end{cor}

\begin{proof}
    Both claims follow from Proposition~\ref{prop: p(G; z) with no even cycles}, Lemma~\ref{lem: no even cycles}, and the recognition that $p(C_n; z)$ differs from $m(C_n; z)$ only in the top coefficient.
\end{proof}

See Table~\ref{tab: paths and cycles} for polynomials $p(P_n; z)$ and $p(C_n; z)$ for small $n$.

\begin{table}
    \begin{center}
    \begin{tabular}{|c|l|l|} \hline
        $n$ & $p(P_n; z)$ & $p(C_n; z)$ \\ \hline
        $1$ & $1$ &  - \\
        $2$ & $1+z$ & - \\
        $3$ & $1+2z$ & $1 + 3z$ \\
        $4$ & $1+3z+z^2$ & $1+4z+z^2$ \\
        $5$ & $1+4z+3z^2$ & $1 + 5z + 5z^2$ \\ 
        $6$ & $1+5z+6z^2+z^3$ & $1+6z+9z^2+z^3$ \\
        $7$ & $1+6z + 10z^2 + 4z^3$ & $1+7z+14z^2+7z^3$ \\
        $8$ & $1+7z+15z^2+10z^3+z^4$ & $1+8z+20z^2+16z^3+z^4$ \\ 
        $9$ & $1+8z+20z^2+16z^3+z^4$ & $1+9z+27z^2+30z^3+9z^4$\\ \hline
    \end{tabular}\vspace{1mm}
    \end{center}
    \caption{The perfectly matchable set polynomials for $P_n$ and $C_n$ with $n \leq 7$.}\label{tab: paths and cycles}
\end{table}

Although Proposition~\ref{prop: p(G; z) with no even cycles} assumes that $G$ has no even cycles, the recurrences do hold for a much larger class of graphs.

\begin{thm}\label{thm: perf match sets recurrence}
    Suppose $G$ is a graph with an edge $e = vw$.
    If $e$ is contained in no even cycle, then
    \begin{equation}\label{eq: main recurrence}
        p(G; z) = p(G \setminus e; z) + zp(G - \{v,w\}; z).
    \end{equation}
    Consequently, if $v$ is a vertex of $G$ contained in no even cycle and has neighbors $w_1,\dots,w_r$, then
	\begin{equation}\label{eq: main recurrence 2}
		p(G; z) = p(G - v; z) + z\sum_{i=1}^r p(G - \{v,w_i\}; z).
	\end{equation}
\end{thm}

\begin{proof}
    First we observe that if $S$ is a perfectly matchable set of $G$, then either every perfect matching of $G[S]$ contains $e$ or none of them contain $e$.
    Indeed, if $\calM$ and $\calM'$ were perfect matchings of $G[S]$ so that $e \in \calM$ but $e \notin \calM'$, then, as in the proof of Lemma~\ref{lem: no even cycles}, $e$ must be part of an even cycle that uses a subset of the vertices in $S$.
    This is impossible since $e$ is contained in no even cycle.
    
    This leads us to define two sets: $I$, the collection of perfectly matchable sets $S$ of $G$ such that no perfect matching of $G[S]$ contains $e$, and $E$, the collection of perfectly matchable sets of $S$ such that every perfect matching of $G[S]$ contains $e$.
    By the argument in the previous paragraph, $I$ and $J$ are well-defined, and their union is the collection of all perfectly matchable sets of $G$.
    
    It is straightforward to see that $I$ is, in fact, exactly the collection of perfectly matchable sets in $G \setminus e$. 
    Additionally, $J$ is in bijection with the perfectly matchable sets of $G - \{v,w\}$ via the map $S \mapsto S - \{v,w\}$, the details of which we leave to the reader.
    Therefore,
    \[
        \begin{aligned}
            p(G;z) &= \sum_{S \in I} z^{|S|/2} + \sum_{S \in J} z^{|S|/2} \\
                &= p(G \setminus e;z) + z\sum_{S \in J} z^{|S - \{v,w\}|/2} \\
                &= p(G \setminus e; z) + zp(G - \{v,w\}; z),
        \end{aligned}
    \]
    as claimed.

    To establish \eqref{eq: main recurrence 2}, recognize that since $v$ is contained in no even cycle, then neither is any edge containing $v$.
    The equation follows by iteratively applying \eqref{eq: main recurrence} to the edges of $G$ containing $v$.
\end{proof}

The strongest benefit to Theorem~\ref{thm: perf match sets recurrence} is that it holds even if $G$ has even cycles, as long as $e$ is not a part of one.

Through the remainder of this section, we will use the following auxiliary function: given a graph $G$ and a vertex $v$, set 
\[
    p_v(G; z) = p(G; z) - p(G - v; z) = \sum_{k = 2}^{\left\lfloor \frac{|V(G)|}{2} \right\rfloor} q_{2k}(v)z^k
\]
where $q_{2k}(v)$ is the number of $2k$-element perfectly matchable subsets of $V(G)$ which contain $v$.
More generally, if $S \subseteq V(G)$, we let $p_S(G;z)$ be the generating function for the number of perfectly matchable sets of $G$ containing every element of $S$.
This satisfies the recursively-defined inclusion-exclusion formula
\[
    p_S(G;z) = \sum_{T \subseteq S} (-1)^{|T|}p(G - T;z).
\]
Whenever $S$ contains at most $2$ elements, we will write those elements without the braces in the subscript of $p$.

Note that if $v$ has degree $1$, then any perfectly matchable set $S$ containing $v$ must also contain its neighbor.
All other elements of $S$ must also be a perfectly matchable set, since a perfect matching $\calM$ using the vertices in $S$ requires $v$ to be matched with its neighbor.
Thus, no other edges incident to the neighbor of $v$ may appear in $\calM$.
We record this observation for later reference.

\begin{lem}
    Let $G$ be a graph with a vertex $v$.
    If $\deg(v)=1$ with unique neighbor $w$, then
    \[
        p_v(G; z) = zp(G - \{v,w\}; z).
    \]
\end{lem}

Our next significant result of the section is the following.

\begin{thm}\label{thm: cut vertex}
    Suppose $G$ is a graph with a vertex $u$ and $H$ is a graph with a vertex $w$.
    Let $G'$ be the graph obtained from the disjoint union $G + H$ by identifying $u$ and $w$.
    Then
    \[
        p(G'; z) = p(G; z)p(H; z) - p_u(G; z)p_w(H; z).
    \]
\end{thm}

\begin{proof}
    We will let $v$ denote the new vertex of $G'$ obtained from identifying $u$ and $w$.
    Let $\PM(\cdot)$ denote the set of perfectly matchable sets of a graph. Consider the function
    \[
        f: \PM(G) \times \PM(H) \to 2^{V(G')}
    \]
    defined by
    \[
        f(S,T) = \begin{cases}
            S \cup T & \text{ if } u \notin S \text{ and } w \notin T \\
            (S \cup T \cup \{v\}) - \{u,w\} & \text{ otherwise}.
        \end{cases}
    \]
    
    First we show that $\PM(G') \subseteq \img(f)$.
    If $X \in \PM(G')$, then let $\calM$ be a perfect matching of $X$ in $G'$.
    Denote by $\calM_G$ the edges in $\calM$ having at least one endpoint in $V(G) - \{u\}$ and, similarly, denote by $\calM_H$ the edges in $\calM$ having at least one endpoint in $V(H) - \{w\}$.
    Let $X_G$ be the subset of $X$ containing all vertices of edges in $\calM_G$, and define $X_H$ analogously. 
    Now, set $\widehat{X}_G$ to be $X_G$ if $v \notin X_G$ and $(X_G - \{v\}) \cup \{u\}$ otherwise.
    Similarly, set $\widehat{X}_H$ to be $X_H$ if $w \notin X_H$ and $(X_H - \{w\}) \cup \{u\}$ otherwise.
    It then immediately follows that $\widehat{X}_G$ is a perfectly matchable set of $G$ and $\widehat{X}_H$ is a perfectly matchable set of 
    $H$ satisfying 
    \[
        X = X_G \cup X_H =  f(\widehat{X}_G,\widehat{X}_H) \in \img(f)
    \]
    as needed.
    Thus, we only need to account for the number of sets in $\img(f)$ which are not a perfectly matchable set of $G'$.
    
    We observe that the restriction of $f$ to the domain $\PM(G - \{u\}) \times \PM(H - \{w\})$
    clearly induces a bijection between pairs $(S,T)$ of perfectly matchable sets satisfying $u \notin S$ and $w \notin T$ with perfectly matchable sets of $G'$ which do not contain $v$.
    So, the only work to be done is when considering pairs $(S,T) \in \PM(G) \times \PM(H)$ where at least one of $u \in S$ or $w \in T$ holds.
    
    Suppose that $X \in \PM(G')$ and $v \in X$.
    Then $|X \cap V(G)|$ is even or $|X \cap V(H)|$ is even, but not both since $v \notin V(G),V(H)$. 
    Without loss of generality, suppose $|X \cap V(H)|$ is even.
    Thus, any perfect matching $\calM$ of $G'[X]$ must have an edge where one endpoint is $v$ and the other endpoint is in $V(G)$. 
    Defining $\widehat{X}_G$ and $\widehat{X}_H$ as before, we find that $X = f(\widehat{X}_G, \widehat{X}_H)$, and -- importantly -- this is the only pair mapping to $X$.
    
    Therefore, any elements of $\img(f)$ not in $\PM(G')$ must arise if both $u \in S$ and $w \in T$.
    When this happens $|f(S,T)|$ is odd, and cannot map to a perfectly matchable set of $G'$.
    
    Passing to the level of perfectly matchable set polynomials gives us the claimed result. 
\end{proof}

In any graph which can be formed as in $G'$ of Theorem~\ref{thm: cut vertex}, the vertex $v$ is said to be a cut-vertex of $G'$.
More precisely, a vertex of a graph $G$ is a \emph{cut-vertex} if deleting it produces a graph with more components. 
Thus, to compute $p(G; z)$ for an arbitrary $G$, it is enough to consider graphs without cut-vertices, also called \emph{$2$-connected graphs}.

Recall that an \emph{ear} of a graph $G$ is a maximal path of $G$ in which every internal vertex has degree $2$ in $G$.
An ear is \emph{closed} if the endpoints of the ear are the same, and is \emph{open} otherwise.
An elementary result of graph theory is that a graph is $2$-connected if and only if it possesses an \emph{open ear decomposition}, which is a sequence of subgraphs $G_0, G_1,\dots, G_r$ such that $G_0$ is a cycle and $G_i$ is an ear of $G_0 \cup \cdots \cup G_{i-1}$ for each $i \geq 1$.
Since we now know how to compute $p(P_n; z)$ and $p(C_n; z)$ for all $n$, we wish to determine a recurrence for $p(G; z)$ in terms of an open ear of $G$.
The primary proof technique used to establish Theorem~\ref{thm: cut vertex} can be extended to address open ears as well.

\begin{cor}\label{cor: attaching open ear}
    Suppose $G'$ is a graph, $P$ is an open ear of $G'$ with endpoints $v$ and $w$, and $G$ is the graph obtained by deleting the internal vertices of $P$.
    The perfectly matchable set polynomial of $G'$ is
    \[
        \begin{aligned}
            p(G'; z) &= p(G; z)p(P; z) - p_v(G; z)p_v(P; z) - p_w(G; z)p_w(P; z) \\
                &{\quad} + p_{v,w}(G; z)p_{v,w}(P; z)  - \sum_{(S,T)} z^{|S \cup T|/2}
        \end{aligned}
    \]
    where the sum is over all pairs $(S,T) \in 2^{V(G)} \times 2^{V(P)}$ satisfying exactly one of the following sets of conditions:
    \begin{enumerate}[label=(\alph*)]
        \item $v,w \in S$ and $v,w \notin T$;
        \item $S$ and $(S - \{v,w\})$ are perfectly matchable sets of $G$; and 
        \item $T$ and $T \cup \{v,w\})$ are perfectly matchable sets of $P$
    \end{enumerate}
    or
    \begin{enumerate}[label=(\alph*$'$)]
        \item $v \in S$, $w \notin S$, $v \notin T$, and $w \in T$;
        \item $S$ and $(S - \{v\}) \cup \{w\}$ are perfectly matchable sets of $G$; and 
        \item $T$ and $(T - \{w\}) \cup \{v\}$ are perfectly matchable sets of $P$.
    \end{enumerate}
\end{cor}

Before proving this result, we give an example to illustrate.

\begin{example}\label{ex: open ear}
    Consider the graph $G'$ illustrated in Figure~\ref{fig: G'}.
    Let $G$ be the graph induced by $[6]$, let $P$ be the open ear induced by $\{2,7,8,9,6\}$, and let $v=2$ and $w=6$.
    
    \begin{figure}
        \centering
        \includegraphics[scale=0.3]{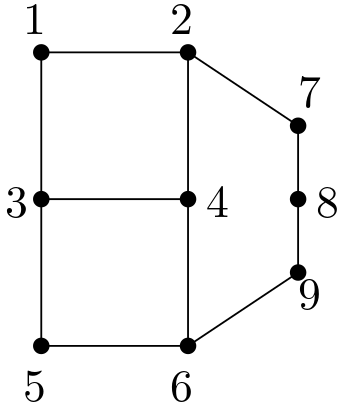}
        \caption{The graph $G'$ used in Example~\ref{ex: open ear} to illustrate Corollary~\ref{cor: attaching open ear}.}\label{fig: G'}
    \end{figure}
    
    Next, we gather the necessary polynomials in the statement of Corollary~\ref{cor: attaching open ear}, which the reader may verify directly:
    \begin{itemize}
        \item $p(G; z) = 1 + 7z + 9z^2 + z^3$;
        \item $p(P; z) = 1 + 4z + 3z^2$;
        \item $p_v(G; z) = p_w(G; z) = 2z + 6z^2+z^3$;
        \item $p_v(P; z) = p_w(P; z) = z+2z^2$;
        \item $p_{v,w}(G; z) = 3z^2+z^3$; and
        \item $p_{v,w}(P; z) = z^2$.
    \end{itemize}
    
    We now need to determine all pairs $(S,T)$ satisfying the specified conditions.
    In this particular example, there are no pairs satisfying (a), (b), and (c).
    For the second set of conditions, we must first have $2 \in S$ and $6 \in T$.
    Since $(S - \{2\}) \cup \{6\}$ must a perfectly matchable set of $G$, $S$ must contain at least one vertex of $G$ that is adjacent to $6$.
    This leads to only four possible choices for $S$: $\{2,4\}$,  $\{1,2,3,4\}$, $\{1,2,3,5\}$ and $\{2,3,4,5\}$. 
    For each of these, replacing $2$ with $6$ will result in another perfectly matchable set of $G$, so these are the four sets $S$ appearing in the final sum.
    The analogous construction for $T$ results in the single perfectly matchable set $\{6,7,8,9\}$.
    Passing to the level of $z^{|S \cup T|/2}$, these four pairs mean we must subtract $z^3 + 3z^4$.
    Applying Corollary~\ref{cor: attaching open ear}, we obtain
    \[
        p(G'; z) = 1+11z+36z^2+37z^3+5z^4.
    \]
\end{example}

\begin{proof}[Proof of Corollary~\ref{cor: attaching open ear}]
    The terms $p(G; z)p(P; z)$, $-p_v(G; z)p_v(P; z)$, and $-p_w(G; z)p_w(P; z)$ arise exactly as in the proof of Theorem~\ref{thm: cut vertex}.
    The term $p_{v,w}(G; z)p_{v,w}(P; z)$ arises by inclusion-exclusion, that is, we must avoid subtracting summands twice when they correspond to pairs of perfectly matchable sets of $G$ and $P$, both containing both $v$ and $w$.

    If there are any summands $z^{|X|/2}$ of $p(G; z)p(P; z)$ which remain to be subtracted, they each must correspond to a perfectly matchable set $X$ of $G'$ which can be written as two distinct disjoint unions $X = S \uplus T$ and $X = S' \uplus T'$ of perfectly matchable sets in $G$ and $P$, respectively.
    By a parity argument, $v,w \in X$.
    Moreover, $v$ is in exactly one of $S$ and $S'$, and the same is true of $w$ and $T,T'$.
    Specifically, the two possibilities are:
    \begin{itemize}
        \item $v,w \in S$ and $v,w \notin S'$ (or vice versa); or 
        \item $v \in S, T'$ and $w \in S',T$ (or vice versa).
    \end{itemize}
    Again by a parity argument, it is not possible for there to be a perfectly matchable set where $S \uplus T$ falls into the first case and $S' \uplus T'$ falls into the second case.
    
    Since each of these cases corresponds to a summand $z^{|X|/2}$, we must subtract exactly one of them. 
    These are exactly the summands subtracted satisfying (a), (b), and (c) or (a$'$), (b$'$), and (c$'$).
\end{proof}

It follows that Corollaries~\ref{cor: paths and cycles} and~\ref{cor: attaching open ear} can be used in conjunction to produce $p(G; z)$ for any $G$.
Therefore, it is possible to compute $p(G; z)$ for arbitrary graphs through these recurrences.

\subsection{In relation to Ehrhart theory}\label{sec: relation to Ehrhart}

Our first result of this section is Lemma~\ref{lem: h star and p of complement}.
Its proof relies on the notion of an edge polytope:
given a graph $G$ on $[n]$, the corresponding \emph{edge polytope} is
\[
    Q_G = \conv\{e_i + e_j \mid ij \in E(G)\}.
\]
The proof also relies on a combination of results in \cite{KalmanPostnikov, OhsugiTsuchiyaInterior}.
We rephrase them slightly in order to avoid discussion of interior polynomials of hypergraphs, which we will not need explicitly.

\begin{thm}[{see \cite[Propositions 3.3 and 3.4]{OhsugiTsuchiyaInterior}}]\label{thm: interior polynomial}
    Suppose $G$ is a bipartite graph with vertex bipartition $(X,Y)$.
    Let $\widehat{G}$ be the graph obtained from $G$ by introducing two new vertices $v$ and $w$, and forming the edges $vy$ for each $y \in Y$ and $xw$ for each $x \in X \cup \{v\}$.
    Then $h^*(Q_{\widehat{G}};z) = p(G;z).$ \qed
\end{thm}

We may now state and prove our first result of this section.

\begin{lem}\label{lem: h star and p of complement}
	Let $G$ be a graph.
	If $\overline{G}$ is bipartite, then $h^*(P_G; z) = p(\overline{G}; z)$.
\end{lem}

\begin{proof}
    Let $X,Y$ be a bipartition of the vertices of $\overline{G}$.
    Form the graph $\widehat{\overline{G}}$ from $\overline{G}$ as in the statement of Theorem~\ref{thm: interior polynomial}, which introduces the new vertices $v$ and $w$.
    
    Consider the edge polytope $Q = Q_{\widehat{\overline{G}}}$.
    Note that $Q$ satisfies the linear equalities
    \[
        \sum_{i \in X \cup \{v\}} x_i = 1 \quad\text{ and }\quad \sum_{i \in Y \cup \{w\}} x_i = 1,
    \]
    from which it follows that $Q$ is lattice-equivalent to $P_G$ via dropping the coordinates corresponding to $v$ and $w$.
    Therefore, $h^*(P_G;z) = h^*(Q;z)$, and by Theorem~\ref{thm: interior polynomial}, we have
    \[
        h^*(P_G;z) = h^*(Q_{\widehat{\overline{G}}};z) = p(\overline{G};z). \qedhere
    \]
\end{proof}

At this point, because we will primarily be considering $h^*$-polynomials of polytopes associated to complements of graphs, we use the following to simplify our notation.

\begin{notation}
    For a graph $G$, set $\h^*(G; z) = h^*(P_{\overline{G}}; z)$.
\end{notation}

Putting Lemma~\ref{lem: no even cycles} together with Proposition~\ref{prop: p(G; z) with no even cycles}, we deduce the following.

\begin{cor}
    If $T$ is a tree and $e = uv$ is any edge, then
    \begin{equation}\label{eq: tree recurrence}
        \pushQED{\qed}
    	\h^*(T; z) = \h^*(T \setminus e; z) + z\h^*(T - \{u,v\}; z). \qedhere\popQED
    \end{equation}
\end{cor}

We can even obtain an explicit for the $\h^*$-polynomial for even cycles.

\begin{cor}\label{cor: h star for odd cycles}
    For all $k \geq 0$,
    \[
        \h^*(C_{2k}; z) = \frac{(1+2z+\sqrt{1+4z})^k + (1+2z-\sqrt{1+4z})^k}{2^k} - z^k.
    \]
    Consequently, $\NVol(P_{C_{2k}}) = \lfloor \varphi^{2k} \rfloor$ where $\varphi = (1+\sqrt{5})/2.$
\end{cor}

\begin{proof}
    By Corollary~\ref{cor: paths and cycles} we deduce that $p(C_{2k}; z)$ satisfies the recurrence
    \[
        p(C_{2k}; z) = (1+2z)p(C_{2k-2}; z) - z^2p(C_{2k-4}; z) + z^{k-1}
    \]
    with initial conditions $p(C_1; z) = 1$ and $p(C_3; z) = 1+z$.
    Thus, the generating function, with respect to a new variable $t$, is
    \[
        \frac{1 - 2zt + (z+z^2)t^2}{(1-(1+2z)t + z^2t^2)(1-zt)}.
    \]
    The given formula for $\h^*(C_{2k}; z)$ may then be extracted from this rational function using elementary algebra, and it follows that $\NVol(P_{C_{2k}}) = \h^*(C_{2k};1) = \lfloor \varphi^{2k} \rfloor$.
\end{proof}


Corollary~\ref{cor: h star for odd cycles} genuinely does require the cycle to be even: see Table~\ref{tab: h star for cycles} for $\h^*(C_n; z)$ for all $3 \leq n \leq 8$, and compare with the second column of Table~\ref{tab: paths and cycles}.
Indeed, when $n=5$, for example, $\h^*(C_5; z) = 1+5z+5z^2+z^3$ cannot be a perfectly matchable set polynomial for any graph with five vertices, although it is the perfectly matchable set polynomial for the tree on six vertices with degree sequence $(3,2,2,1,1,1)$. 
That said, it appears that, when $n \geq 5$ is odd, $\h^*(C_n; z) = p(C_n; z) + z^{(n+1)/2}$.
This allows us to state the following conjecture.

\begin{table}
    \centering
    \begin{tabular}{|c|l|} \hline
        $n$ & $\h^*(C_n; z)$ \\ \hline
        $3$ & $1 + 4z + z^2$ \\
        $4$ & $1 + 4z + z^2$ \\
        $5$ & $1 + 5z + 5z^2  + z^3$ \\
        $6$ & $1 + 6z + 9z^2  + z^3$ \\
        $7$ & $1 + 7z + 14z^2 + 7z^3  + z^4$ \\
        $8$ & $1 + 8z + 20z^2 + 16z^3 + z^4$ \\
        $9$ & $1 + 9z + 27z^2 + 30z^3 + 9z^4 + z^5$ \\ \hline
    \end{tabular}
    \caption{The polynomials $\h^*(C_n; z)$ for all $3 \leq n \leq 9$.}
    \label{tab: h star for cycles}
\end{table}

\begin{conj}
    For all $n \geq 4$,
    \[
        \h^*(C_n; z) = p(C_n; z) + \frac{1+(-1)^{n+1}}{2}z^{(n+1)/2}.
    \]
\end{conj}

\section{Special cases for trees}\label{sec: special cases}

In this section, we conduct a more specific analysis of $\h^*(T; z)$ when $T$ is a tree.
Recall from Corollary~\ref{cor: paths and cycles} and Lemma~\ref{lem: no even cycles} that $\h^*(P_n; z) = u_{n+1}(1,z)$, in which case $\NVol(P_{\overline{P}_n}) = F_{n+1}$, the $(n+1)^{th}$ Fibonacci number. 
We explore this and similar phenomena further for various classes of trees.

Recall that a \emph{caterpillar} is a tree consisting of a path, called the \emph{spine}, and a set of leaves that are each adjacent to a vertex of the spine.
We denote by $\cat(n,k)$ the caterpillar whose spine has $n$ vertices and every spine vertex is adjacent to exactly $k$ leaves.
See Figure~\ref{fig: caterpillar} for an illustration of $\cat(4,3)$.
    
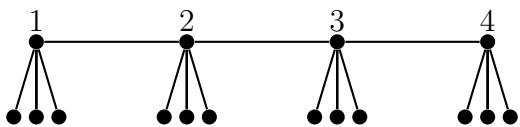
\begin{figure}
\begin{center}
\begin{tikzpicture}
\begin{scope}[every node/.style={circle,fill,inner sep=0pt,minimum size=2mm}]
	\node[label=$1$] (A) at (-3,0) {};
	\node[label=$2$] (B) at (-1,0) {};
	\node[label=$3$] (C) at (1,0) {};
	\node[label=$4$] (D) at (3,0) {};
	
	\node (A1) at (-3.3,-1) {};
	\node (A2) at (-3,-1) {};
	\node (A3) at (-2.7,-1) {};
	
	\node (B1) at (-1.3,-1) {};
	\node (B2) at (-1,-1) {};
	\node (B3) at (-0.7,-1) {};
	
	\node (C1) at (0.7,-1) {};
	\node (C2) at (1,-1) {};
	\node (C3) at (1.3,-1) {};
	
	\node (D1) at (2.7,-1) {};
	\node (D2) at (3,-1) {};
	\node (D3) at (3.3,-1) {};
\end{scope}

\draw[thick] (A) -- (B) -- (C) -- (D);
\draw[thick] (A1) -- (A) -- (A2) -- (A) -- (A3);
\draw[thick] (B1) -- (B) -- (B2) -- (B) -- (B3);
\draw[thick] (C1) -- (C) -- (C2) -- (C) -- (C3);
\draw[thick] (D1) -- (D) -- (D2) -- (D) -- (D3);

\end{tikzpicture}
\end{center}
\caption{The caterpillar $\cat(4,3)$. The spine consists of the graph induced by the vertex set $\{1,2,3,4\}$.}\label{fig: caterpillar}
\end{figure}

\begin{prop}\label{prop: caterpillars}
    For all $n,k \geq 0$, $\h^*(\cat(n,k); z)$ is given by
    \[
         \frac{(1+kz+\sqrt{1+2(2+k)z+k^2z^2})^{n+1} - (1+kz-\sqrt{1+2(2+k)z+k^2z^2})^{n+1}}{2^{n+1}\sqrt{1+2(2+k)z+k^2z^2}}.
	\]
\end{prop}
    
\begin{proof}
    Applying Theorem~\ref{thm: perf match sets recurrence} on one of the last edges of the spine, we have 
    \[
	    \h^*(\cat(n,k); z) = (1+kz)\h^*(\cat(n-1,k); z) + z\h^*(\cat(n-2,k); z)
	\]
	with initial conditions $1$ and $1+kz$ for $n=0$ and $n=1$, respectively.
	From this recurrence we obtain the generating function
	\[
	    \frac{1}{1-(1+kz)t-zt^2}
	\]
	for $\h^*(\cat(n,k))$, and we extract the exact formula using elementary algebra.
\end{proof}

Note that we recover $\NVol(P_{\overline{P}_n}) = F_{n+1}$ by setting $k=0$ and $z=1$ in Proposition~\ref{prop: caterpillars}, but the proposition gives us more refined information.
In fact, Proposition~\ref{prop: caterpillars} leads to a simple asymptotic formula for the normalized volume of $P_{\overline{\cat{n,k}}}$ as the number of leaves increases.
    
\begin{cor}
    For a fixed positive integer $n$, $\NVol(P_{\overline{\cat(n,k)}}) \to k^n$ as $k \to \infty$. \qed
\end{cor}

It is worth stating here in detail how the $\h^*$-polynomials in Proposition~\ref{prop: caterpillars} relate to existing generalizations of Fibonacci numbers.
Recall from Section~\ref{sec: recurrences} the functions $u_n(x,y)$, defined for all $n \geq 0$ with initial conditions $u_0(x,y) = 0$ and $u_1(x,y) = 1$, and
\[
    u_n(x,y) = xu_{n-1}(x,y) + yu_{n-2}(x,y)
\]
for $n \geq 2$.
The most common definition of a Fibonacci polynomial is the polynomial $F_n(x) = u_n(x,1)$, which satisfies the recurrence
\[
        F_n(x) = xF_{n-1}(x) + F_{n-2}(x)
\]
for $n \geq 2$.
Proposition~\ref{prop: caterpillars}, together with Corollary~\ref{cor: paths and cycles}, shows that $\h^*(P_n; z) = u_{n+1}(1,z)$, yields a different polynomial which specializes to a Fibonacci number when setting $z=1$.
These two polynomials are related via $F_n(x) = \h^*(P_{n-1};x^2)$ when $n$ is odd and $F_n(x) = x\h^*(P_{n-1};x^2)$ when $n$ is even, but only one appears in our discussion as an $h^*$-polynomial for a graph.
The polynomials $u_n(1,z)$ are also sometimes called \emph{Jacobsthal polynomials} e.g. in \cite[Chapter 39]{Koshy} and \cite{Hoggatt-BicknellJohnson} and the references therein, although it has become more common to define Jacobsthal polynomials as $u_n(1,2z)$.

Caterpillars themselves hold yet more information.
For example, when considering $\cat(n,1)$, we obtain a polynomial analogue of the \emph{Pell numbers}: the sequence obtained from initial terms $1$ and $3$ and following the same recurrence as the Fibonacci numbers.
    
\begin{cor}\label{cor: pell}
    For all $n \geq 2$, 
    \[
        \h^*(\cat(n,1); z) = (1+z)\h^*(\cat(n-1,1); z) + z\h^*(\cat(n-2,1); z),
	\]
    with the initial conditions $\h^*(\cat(0,1); z) = 1$ and $\h^*(\cat(1,1); z) = 1+z$. \qed
\end{cor}
	
See Table~\ref{tab: pell} for $\h^*(\cat(n,1); z)$ for small values of $n$.
Such polynomial analogues of Pell numbers have previously been studied; see \cite[Sequence A037027]{OEIS} for connections involving lattice paths, pattern avoidance in ternary words, and more. 
	
\begin{table}
    \begin{center}
    \begin{tabular}{|c|l|} \hline
        $n$ & $\h^*(\cat(n,1); z)$ \\ \hline
        $0$ & $1$  \\
        $1$ & $1+z$ \\
        $2$ & $1+3z+z^2$ \\
        $3$ & $1+5z+5z^2+z^3$ \\
        $4$ & $1+7z+13z^2+7z^3+z^4$ \\
        $5$ & $1+9z+25z^2+25z^3+9z^4+z^5$ \\ 
        $6$ & $1+11z+41z^2+63z^3+41z^4+11z^5+z^6$ \\
        $7$ & $1+13z+61z^2+129z^3+129z^4+61z^5+13z^6+z^7$ \\ \hline
    \end{tabular}\vspace{1mm}
    \end{center}
    \caption{The $h^*$-polynomials $\h^*(\cat(n,1); z)$ for $n \leq 7$. 
    Each of these is a polynomial analogue of a Pell number.}\label{tab: pell}
\end{table}

Countless other sequences may be given polynomial analogues as $h^*$-polynomials of lattice polytopes, now, due to the recurrence \eqref{eq: tree recurrence}. 
We close this section by highlighting one more special case: when $T$ is a complete $k$-ary tree of rank $r$.
Recall that a \emph{complete $k$-ary tree of rank $r$} is a rooted tree in which every internal vertex has $k$ descendants, and every path from the root to a leaf has $r+1$ vertices. 
For $k, r \geq 0$, we denote by $\Lambda_{k,r}$ the complete $k$-ary tree of rank $r$.

\begin{prop}
    For all $k, r \geq 0$,
	\[
	    \h^*(\Lambda_{k,r}; z) = \h^*(P_{r+1};kz)\prod_{i=1}^{r-1} \h^*(\Lambda_{k,i}; z)^{k-1}.
	\]
	Consequently, $h^*(\Lambda_{k,r}; z)$ may be expressed as a product of the polynomials $\h^*(P_n; \ell z)$ for certain choices of $n$ and $\ell$.
\end{prop}

\begin{proof}
    We prove this via induction on $r$.
    Verifying the claim for $r=0,1$ is straightforward and therefore omitted.
    For the induction step, we combine Lemma~\ref{lem: no even cycles} and Proposition~\ref{prop: p(G; z) with no even cycles}, applied to the root of $\Lambda_{k,r}$, to obtain
    \[
        \begin{aligned}
            \h^*(\Lambda_{k,r}; z) &= \h^*(\Lambda_{k,r-1}; z)^k + kz\h^*(\Lambda_{k,r-2}; z)^k\h^*(\Lambda_{k,r-1}; z)^{k-1} \\
            &= \h^*(\Lambda_{k,r-1}; z)^{k-1}\left(\h^*(\Lambda_{k,r-1}; z) + kz\h^*(\Lambda_{k,r-2}; z)^k\right)
        \end{aligned}
    \]
    for $r \geq 2$.
    By applying the inductive assumption and factoring, the right side becomes
    \[
      \resizebox{.9\hsize}{!}{$\displaystyle{\h^*(\Lambda_{k,r-1}; z)^{k-1}\prod_{i=1}^{r-3} \h^*(\Lambda_{k,i}; z)^{k-1}\left(\h^*(P_r;kz)\h^*(\Lambda_{k,r-2}; z)^{k-1} + kz\h^*(P_{r-1};kz)^k\prod_{i=1}^{r-3} \h^*(\Lambda_{k,i}; z)^{(k-1)^2}\right)}.$}
    \]
    Notice that
    \[
        \h^*(P_{r-1};kz)^k\prod_{i=1}^{r-3} \h^*(\Lambda_{k,i}; z)^{(k-1)^2} = \h^*(P_{r-1};kz)\h^*(\Lambda_{k,r-2}; z)^{k-1}.
    \]
    Applying this and simplifying more, we obtain
    \[
        \begin{aligned}
            \h^*(\Lambda_{k,r}; z) &= \prod_{i=1}^{r-1} \h^*(\Lambda_{k,i}; z)^{k-1}\left(\h^*(P_r;kz) + kz\h^*(P_{r-1};kz)\right) \\
            &= \h^*(P_{r+1};kz)\prod_{i=1}^{r-1} \h^*(\Lambda_{k,i}; z)^{k-1},
        \end{aligned}
    \]
    as claimed.
\end{proof}


\bibliographystyle{plain}
\bibliography{references}

\end{document}